\documentclass[english,12pt]{article}
\usepackage[T1]{fontenc}
\usepackage[latin9]{inputenc}
\usepackage{geometry}
\usepackage{color}
\usepackage{amsthm}
\usepackage{amsmath}
\usepackage{graphicx}
\usepackage{amssymb}
\usepackage{esint}

 \usepackage{amsfonts}
 
\newtheorem{assumption}{Assumption}[section]
\theoremstyle{plain}
\theoremstyle{plain}
\newtheorem{thm}{Theorem}
  \theoremstyle{definition}
  \newtheorem{defn}[thm]{Definition}
  \theoremstyle{plain}
  \newtheorem{lem}[thm]{Lemma}

\DeclareMathOperator*{\argmin}{argmin}

\AtBeginDocument{
  
}

\usepackage{babel}

\begin{document}

\title{Delayed Feedback Control Requires an Internal Forward Model}

\author{Dmitry Volkinshtein and Ron Meir\\
Department of Electrical Engineering\\
Technion, Haifa 32000, Israel{\footnotesize }\\
{\footnotesize dmitryvolk@gmail.com; rmeir@ee.technion.ac.il}\\
}
\maketitle
\begin{abstract}
Biological motor control provides highly effective solutions to difficult
control problems in spite of the complexity of the plant and the significant
delays in sensory feedback . Such delays are expected to lead to non
trivial stability issues and lack of robustness of control solutions.
However, such difficulties are not observed in biological systems
under normal operating conditions. Based on early suggestions in the
control literature, a possible solution to this conundrum has been
the suggestion that the motor system contains within itself a forward
model of the plant (e.g., the arm), which allows the system to `simulate'
and predict the effect of applying a control signal. In this work
we formally define the notion of a forward model for deterministic
control problems, and provide simple conditions that imply its existence
for tasks involving delayed feedback control. As opposed to previous
work which dealt mostly with linear plants and quadratic cost functions,
our results apply to rather generic control systems, showing that
any controller (biological or otherwise) which solves a set of tasks,
\emph{must} contain within itself a forward plant model. We suggest
that our results provide strong theoretical support for the necessity
of forward models in many delayed control problems, implying that
they are not only useful, but rather, mandatory, under general conditions.
\end{abstract}

\section{Introduction}

The motivation for this work arose from biological motor control,
which is plagued by inherent delays arising in sensory pathways, central
processing units and motor outputs \cite{DavWol05,Kaw99}. However,
the results established shed light on any feedback control system,
subject to observation delays. Such delays, which in primates may
reach 200-300 ms for visually guided arm movements, are very large
compared to fast (150 ms) and intermediate (500 ms) movements \cite{DavWol05,Kaw99},
and may lead to significant difficulties, as inappropriate control
might cause instability or degraded performance. Delays have historically
played a minor role in the field of robotics, as they can usually
be made extremely small in such engineering applications. However,
delayed state feedback has become increasingly important in engineering
fields such as chemical control, distributed system control \cite{Thomopoulos94}
and multisensory tracking \cite{Bar-Shalom02}. In fact, one of the
first attempts within the biological motor control literature \cite{MiaWeiWol93}
to address these issues was based on a well known concept from control
theory, namely the Smith predictor \cite{Smith57}. However, one should
keep in mind that in attempting to understand biological control systems,
based on control theoretic principles, one is in fact trying to `reverse
engineer' an unknown system, as opposed to the task facing an engineer,
namely \emph{designing} a control system (see \cite{WelBulKal08}
for a survey of the possible role of control theory in systems biology).
\\

Within an optimal control based approach, one needs to specify a class
of admissible control laws, a set of plant constraints (e.g., musculo-skeletal),
and a quantitative definition of performance, typically formulated
in terms of a cost function. An optimal control law is then derived
by minimizing a cost function subject to the relevant constraints.
However, within a biological context, the precise nature of the plant
and the controller is seldom known precisely, and the cost function
used by the system (if indeed one is used), may also be unknown. It
would thus be useful to determine general conditions for the necessity
of a forward model, which require as few assumptions as possible.
While a solution to the delay problem in the form of a forward model
is indeed plausible and intuitively appealing \cite{Smith57}, the
question arises as to whether it is \emph{mandatory}, namely, is it
possible to construct an optimal closed-loop control law which is
\emph{not} based on a forward model? As we show in this paper, the
answer to this question is negative, under very mild and reasonable
conditions. More specifically, we show that (under appropriate conditions)
an optimal feedback control law based on \emph{delayed} state observations,
\emph{must} incorporate within itself a forward model of the plant.
\\

As far as we are aware, there is currently no general theory which
provides precise conditions for which forward models are indeed necessary.
Early work, mainly concerned with the linear case (e.g., \cite{Fuller68,Kleinman69,WatIto81}),
suggested several approaches to delayed control problems, including
the proposal that a predictive plant model is needed, as in \cite{Smith57}.
For example, \cite{Kleinman69} showed that optimal control for linear
systems based on minimizing a quadratic cost is obtained by cascading
a Kalman filter and a least-mean square state predictor. Later work
extended these results in various directions. For example, \cite{WatIto81}
suggested an approach to dealing with disturbance attenuation and
\cite{MirRas03}, focusing on stability issues, extended these results
to more general linear systems, showing that state prediction is indeed
a necessary component of such controllers. A survey of many aspects
of this work, circa 2003, appears in \cite{GuNic03}. We note that
much of this work has dealt with the design of actual controllers
(often for linear systems and quadratic cost). As mentioned above,
our perspective in this work is somewhat removed from controller design,
as we are concerned with a \emph{reverse} engineering problem. More
concretely, we begin with an observed control system, operating effectively
under conditions of delayed state observations, and demonstrate that
\emph{any} effective controller \emph{must} contain a forward plant
model. Since it is hard, in general, to make even qualitatively correct
assumptions about the system (e.g., linearity of dynamics and quadratic
cost), we attempt to provide the most general result possible. \\

Before proceeding to a detailed description of our results, we note
that the notion of an internal model has played an important role
in control theory also in other contexts. Francis and Wonham \cite{FraWon75}
were the first to show that stable adaptation (a.k.a. regulation)
requires the existence of an internal model. Adaptation refers to
a situation where the output of the system maintains a constant asymptotic
value whenever the system is subject to inputs from some class of
signals. Intuitively, such an internal model enables the system to
`subtract' external inputs, thereby eliminating their long term effect
on the system. Recently, a powerful extension of this theory was proposed
in \cite{Sontag03}, where it was demonstrated to hold under very
general conditions, without requiring the split into a `plant' and
a `controller' which was required in the original framework of \cite{FraWon75}.
Interestingly, this general theory has been applied to bacterial chemotaxis
and shown to provide interesting novel insight in other biological
situations as well. \\

In summary, our main contribution in this work is the establishment
of precise mathematical conditions for generic deterministic delayed
feedback control systems to possess an internal forward model (we
comment on the extension to the stochastic setting in Section \ref{sec:Discussion},
but leave the full elaboration of this direction to future work).
The generality of the results, and the nature of the conditions required
for them to hold, set the stage for the development, and experimental
verification, of a rigorous theory of delayed feedback control in
biological systems. \\

The remainder of the paper is organized as follows. Section \ref{sec:Results}
presents an overview of the main results, outlining sufficient conditions
for delayed feedback control systems to possess a forward model. Specifically,
in section \ref{sub:Problem-definition} we outline the problem, followed
by a simple example in section \ref{sub:Simple-example} and a summary
of the main results in section \ref{sub:General-results}. In section
\ref{sub:res-Linear-time-optimal} we apply the general ideas to linear
systems with time optimal control and delayed state observations,
while in section \ref{sub:Minimum-Jerk-Optimal} we consider the problem
of minimum jerk control. Section \ref{sec:methods} contains precise
mathematical definitions and full proofs of the main results, including
a full analysis of two examples presented cursorily in section \ref{sec:Results}.
Finally, in section \ref{sec:Discussion} we summarize our results
and present some open research questions. \\

\section{\label{sec:Results}Results}

This section contains a relatively informal summary of our main results.
Precise definitions, assumptions, theorems and proofs appear in section
\ref{sec:methods}. We begin by presenting the problem formulation,
followed by a description of conditions for which a forward model
is \textcolor{black}{mandatory}. We will then use the general necessary
conditions established to show that in linear time optimal control
and minimum jerk optimal control, based on delayed state observations,
a forward model is indeed \textcolor{black}{required}.\\

\subsection{\label{sub:Problem-definition}Problem definition}

\begin{center}
Figure 1 about here
\end{center}

Consider a system to be controlled, referred to as a \emph{plant}.
A plant is usually described by a \emph{state }vector\emph{ }$x^{p}\in X\subseteq\mathbb{R}^{n}$.
For example, in a 2-D motor control setting with joint torques as
control inputs, the plant is a 2-D manipulator. Its state consists
of a pair of joint angles and two velocities. Assuming that the joint
angles take values in the range $[0,\pi]$, while the velocities can
assume any real value, we have $X=[0,\pi]^{2}\times\mathbb{R}^{2}$.
The plant state dynamics are typically given by a differential equation
of the form \begin{equation}
\dot{x}_{t}^{p}=A_{p}(x_{t}^{p},u_{t}),\label{eq:plant_desc}\end{equation}
where $x_{t}^{p}$ is the state of the plant at time $t$, $\dot{x}_{t}^{p}$
denotes the temporal derivative of $x_{t}^{p}$, $u_{t}\in U$ is
the \emph{control} at time $t$, chosen from a set of possible controls
$U$, and $A_{p}$ is a function mapping the state and control to
$\mathbb{R}^{n}$, namely $A_{p}:X\times U\rightarrow\mathbb{R}^{n}$.
In the above example of a 2-D manipulator, assuming that the torque
is bounded in magnitude by $1$, we have $U=[-1,1]^{2}$. \\

In this work we study controllers possessing a memory which, as we
demonstrate, is essential in the case of optimal control with delayed
observations. The memory of the controller at time $t$ can be conceived
of as the controller's state at time $t$. For example, it is well
known \cite{Smith57} that when controlling a plant with delayed state
observation of duration $D$, using the previous controls $\left\{ u_{s}\right\} $
for $t-D\leq s\leq t$ can be useful in order to calculate the current
state of the plant. In this case the controller's memory can be described
by a function $x_{t}^{c}(\cdot)$, where $x_{t}^{c}(\alpha)=u_{t-\alpha}$
for all $0\leq\alpha\leq D$, namely the delayed control.\\

In order to rigorously investigate the notion of a forward model and
derive conditions for its existence, we quantify this notion mathematically
in section \ref{sec:meth-Problem-definition}; here we summarize the
main ideas. In the deterministic delayed state feedback case considered
here, we define a forward model by the ability of the controller to
compute $x_{t}^{p}$, the exact state of the plant at time $t$, given
the delayed observation $x_{t-D}^{p}$ and its memory $x_{t}^{c}(\cdot)$.
This ability to predict the exact state of the plant is equivalent
to the existence of a transformation $F$ such that

\begin{equation}
x_{t}^{p}=F(x_{t}^{c},x_{t-D}^{p})\,\qquad\qquad\textrm{(forward\ensuremath{\,} model)}.\label{eq:forward-model}\end{equation}

In order to clarify the definition, consider a situation when a controller
\emph{does not} possess a forward model. This occurs when the relevant
information available to the controller at time $t$ does \emph{not}
suffice to determine the current plant state $x_{t}^{p}$ unambiguously.
More precisely, based on the current relevant information, $(x_{t-D}^{p},x_{t}^{c})$,
the controller cannot determine $x_{t}^{p}$. Note that the controller
in our model has additional information beyond $(x_{t-D}^{p},x_{t}^{c})$
(see Figure \ref{fig:The-system}); as we claim later, this is irrelevant
to the estimation of the current state $x_{t}^{p}$. \\

The need for a forward model can be established for many scenarios
such as regulation, tracking and optimal control, and the proof is
similar for all. We therefore use a common notion of\emph{ tasks}
to refer to all the above. An example of a task for a 2-D manipulator
is reaching some point $x^{p*}$ on a plane within a prespecified
period of time, or, alternatively, in minimum time. Another possible
task would be holding the manipulator still for $10$ seconds. Clearly,
one can envisage any number of such tasks. The set of all tasks of
interest will be denoted by $X^{*}$. Tasks are fed to the controller
sequentially, and it is assumed that each task can be performed for
each initial state. Note that the system is assumed to be causal,
thus the controller has access only to the current task that should
be performed and not to future tasks. The system described, based
on delayed state observations, is illustrated in figure \textcolor{red}{\ref{fig:The-system}.
}The solving set of control laws for task $x^{*}$, up to time $t$,
is denoted by $U_{t}^{*}(x^{p},x^{*})$ where $x^{p}$ is the initial
state of the plant.\textcolor{red}{}\\

\textcolor{black}{We will show in the sequel that the `richness' of
the set of tasks $X^{*}$, and the corresponding control solutions
$U_{t}^{*}(x^{p},x^{*})$ can make a difference, as to whether a controller
solving the task must possess a forward model or not. For example,
in section \ref{sub:Simple-example} we introduce a plant and a controller
solving a linear time optimal problem. In the first case, where the
set of target states is $X^{*}=[-1,1]$, we show that a forward model
is indeed essential. However, in the case where the set of targets
is limited to two values, $X^{*}=\{-1,1\}$, we give an example of
memoryless controller, which does }\textcolor{black}{\emph{not}}\textcolor{black}{~possess
a forward model, while still solving the optimal control problem perfectly
(i.e., a forward model is not needed in this case).}\textcolor{red}{}\\

Next, we introduce a switching process, $z_{t}$, which defines the
times at which new tasks are specified. Each task is assumed to be
fixed between two consecutive task initiations. A precise definition
of the switching process can be found in section \ref{sec:methods}.
A control law is then defined by \begin{equation}
u_{t}=B_{c}\left(x_{t}^{*},z_{t},x_{t}^{c},\tilde{x}_{t-D}^{p}\right),\label{eq:contr-form}\end{equation}
where $D$ is the observation delay, $x_{t}^{*}\in X^{*}$ is the
task to be performed at time $t$, and $B_{c}$ is a given function.
We have introduced the notation $\tilde{x}_{t-D}^{p}=x_{(t-D)_{+}}^{p},$
where $(x)_{+}=\max(0,x)$, in order to deal systematically with times
$t<D$. In addition to the control signal itself, we consider the
dynamics of the controller's state (memory). One standard formulation
is in terms of a differential equation,\begin{equation}
\begin{cases}
\dot{x}_{t}^{c}(\alpha)=A_{c}\left(x_{t}^{*},z_{t},x_{t}^{c},\tilde{x}_{t-D},\alpha\right)\\
x_{0}^{c}(\alpha)=D_{c}(x_{0}^{p},x_{0}^{*},\alpha)\end{cases},\label{eq:contr-mem-ode}\end{equation}
 where, $A_{c}$ and $D_{c}$ are given functions describing the dynamics
and initial conditions respectively.

In the definition of a forward model, we stated that the relevant
information available to the controller regarding the current state
$x_{t}^{p}$ is $(x_{t-D}^{p},x_{t}^{c})$. The controller has additional
information available at time $t$, consisting of $x_{t}^{*}$ and
$z_{t}$. However, since a new task can be specified at any time (independently
of the value of $x_{t}^{p}$), the current state $x_{t}^{p}$ cannot
depend on these values. \\

\subsection{\label{sub:Simple-example}Example - a simple linear time optimal
control problem }

\begin{center}
Figure 2 about here
\end{center}The abstract ideas introduced in the previous section are clarified
through a simple example. Consider a linear \emph{one dimensional}
time optimal control problem, where the objective is to drive the
plant (described by a single real-valued variable $x^{p}$), to a
point $x^{*}\in X^{*}=X$ in \emph{minimum time}. The plant dynamics
are given by \begin{equation}
\dot{x}_{t}^{p}=-x_{t}^{p}+u_{t}\qquad;\qquad u_{t}\in[-1,1]\,.\label{eq:ex1-ode}\end{equation}
The minimum time cost function is given by \begin{equation}
J(x^{p},x^{*})=\int_{0}^{\tau_{f}}1dt=\tau_{f}\,,\label{eq:simple-ex-cost}\end{equation}
where $\tau_{f}$ is the first time for which $x_{t}^{p}=x^{*}$,
and the initial state of the plant is $x^{p}$. Thus, the controller
needs to minimize $J(x^{p},x^{*})$. The set of tasks here corresponds
to reaching any state $x^{*}\in X$ in minimum time. It is obvious
that if $X=[-1,1]$, all the tasks can be performed, and the optimal
solution in this case is simple and given by $u_{t}=\mathrm{sgn}(x^{*}-x_{t}^{p})$.
This is an example of a so-called bang-bang control, where the control
switches between its extreme allowed values; see figure \ref{fig:Simple-Example}
for a graphical illustration.\\

Before proceeding to establish the existence of a forward model we
summarize the gist of the argument. We start by assuming that a controller
can solve a set of tasks $X^{*}$, based on delayed state observation.
We then argue by contradiction that if the controller lacks a forward
model, then one can find a specific task $x^{*}\in X^{*}$ such that
the controller will not be able to perform the task correctly, in
contradiction to the assumption. Notice that the existence of such
a task is a system related issue that has nothing to do with delays
or a specific {}``black box controller'', as will be explained in
section \ref{sub:General-results} .\\

The argument for the necessity of a forward model in the present example
proceeds as follows (precise statements and proofs appear in sections
\ref{sub:meth-General-Results} and \ref{sub:Linear-time-optimal}).
Assume that we are provided with a black box controller, which performs
the linear time optimal control task optimally, based on \emph{delayed}
state observations. We will show that such a controller must possess
a forward model. Assume to the contrary that it does not, thus there
exist two distinct states, $s^{1}$ and $s^{2},$ $s^{1}\neq s^{2}$,
such that the controller cannot determine whether the plant is currently
in state $x_{t}^{p}=s^{1}$ or $x_{t}^{p}=s^{2}$. In other words,
the controller's available information relevant to the current state,
namely $(x_{t-D}^{p},x_{t}^{c})$, does not suffice to determine $x_{t}^{p}.$
This implies that there exist two trials (namely, two initial states,
times $t_{1}$ and $t_{2}$ and histories of tasks) such that the
available information for both is identical, namely $(x_{t_{1}-D}^{p},x_{t_{1}}^{c})=(x_{t_{2}-D}^{p},x_{t_{2}}^{c})$,
and such that $x_{t_{1}}^{p}=s^{1}$ and $x_{t_{2}}^{p}=s^{2}$ where
$s^{1}\neq s^{2}$. However, if we specify an identical new task at
times $t_{1}$ and $t_{2}$, namely $(x_{t_{1}}^{*},z_{t_{1}})=(x_{t_{2}}^{*},z_{t_{2}})$,
the controller will choose $u_{t_{1}}=u_{t_{2}}$ due to (\ref{eq:contr-form}).
On the other hand, consider the system dynamics (\ref{eq:ex1-ode}),
and choose $x_{t_{1}}^{*}=x_{t_{2}}^{*}=\left(s^{1}+s^{2}\right)/2$,
assuming, without loss of generality, that $s^{1}<s^{2}$. Based on
the exact solution $u_{t}=\textrm{sgn}(x^{*}-x_{t}^{p})$, the optimal
controls are $u_{t_{1}}^{*}=u^{1*}=1$ and $u_{t_{2}}^{*}=u^{2*}=-1$.
However, based on the assumption that the forward model does not exist,
we have shown that $u_{t_{1}}=u_{t_{2}},$ which contradicts the {}``correct
task performing'' assumption. Thus, in this example, a forward model
is indeed required. \\

In order to better understand the requirement for a forward model,
we consider an example where such a model is \emph{not} needed. Consider
the example discussed above, but where the set of tasks (destination
states) consists of only two points $X^{*}=\{-1,1\}$. In this case
a simple memoryless controller such as $u_{t}=x_{t}^{*}$ is optimal,
and clearly lacks a forward model. The reason for this is simple.
When two states $s^{1}\neq s^{2}$ are given, one cannot find a task
$x^{*}$ such that the controls from $s^{1}$ and $s^{2}$ will differ.
The reason is that if $x^{*}=-1$, the controller has to use $u=-1$
and if $x^{*}=1$, it has to choose $u=1$ \emph{independently} of
the initial state. Intuitively, the controller is not required to
know the exact state of the plant in order to be optimal (perform
the task). This simple example and intuition will form the basis of
our general proof in section \ref{sub:meth-General-Results}.

\subsection{General results\label{sub:General-results}}

Having argued for the existence of a forward model in a simple linear
example, we extend the results to a general setting. To do this, we
need to specify when a controller works {}``well''. Such a controller
should perform all possible sequences of tasks correctly, which means
that at each time, $t_{i}$, where a new task is given, the control
signal for the task should belong to the set of controls $U_{t_{i+1}-t_{i}}^{*}(x_{t_{i}}^{p},x_{t_{i}}^{*})$
performing the tasks correctly between the times $t_{i}$ and $t_{i+1}$.
We will refer to such a system as a \emph{Correct Task Performing
System} (CTPS); a precise characterization is provided in definition
\textcolor{red}{\ref{def:optimal-system}}. This definition, based
on the assumption that the task can always be solved, allows one to
build a state feedback controller easily. We show that under these
circumstances, a {}``delayed state feedback controller'' can be
built as well. We refer the reader to Theorem \ref{thm:meth-can-build}
for a precise statement of the result. \\

The proof of Theorem \ref{thm:meth-can-build} is based on building
a controller that uses delayed observations, by defining the memory
of the controller to be $x_{t}^{c}(\alpha)=u_{t-\alpha}$. Then, given
the observation $x_{t-D}^{p}$ , the current state $x_{t}^{p}$ can
be reconstructed by solving the differential equation for the plant
with initial condition $x_{t-D}^{p}$, where the previous controls
are taken from the memory. Once the real state $x_{t}^{p}$ is available,
we can choose the control from the set $U_{t}^{*}$. \\

As demonstrated in the simple example presented in section \ref{sub:Simple-example},
a forward model may not always be necessary. As shown in section \ref{sub:meth-General-Results},
the necessity of a forward model can be demonstrated in situations
where the problem is sufficiently `rich'. In the example above, when
the task set is binary, namely $X^{*}=\left\{ -1,+1\right\} $, no
forward model was required, while if $X^{*}=\left[-1,+1\right]$ a
forward model is indeed required. This idea of problem richness is
formalized in section \ref{sub:meth-General-Results}. We will refer
to a problem as sufficiently `rich' by saying that it does \emph{not}
contain \emph{Non Separable by Correct Task Performing} (NSCTP) pairs
of states; see definition \ref{def:A-pair-of} for a precise characterization.
Intuitively, we say that a pair of states is\emph{ }NSCTP when for
\emph{every} task, the \emph{same} correct control exists at time
$0$ for both states (however, the control may differ for each task).
The main contribution of this paper, Theorem \ref{thm:meth-have-fm-det},
establishes the existence of a forward model when NSCTP pairs of states
do not exist (i.e., the absence of NSCTP pairs of states is a sufficient
condition for a forward model to exist). \\

As a specific illustration of this idea, let us look back at the example
in section \ref{sub:Simple-example}. We implicitly proved there that
the system does not have NSCTP pairs of states by finding a task $x^{*}=\left(s^{1}+s^{2}\right)/2$,
and showing that it leads to $u^{1*}=1$ and $u^{2*}=-1$. The existence
of a forward model in this case (and in more general cases to be studied
in the sequel) follows from theorem \ref{thm:meth-have-fm-det}.

\subsection{\label{sub:res-Linear-time-optimal}Linear time optimal control}

We consider an \emph{optimal setpoint tracking} problem within linear
control theory. The objective here is to reach, from an arbitrary
initial position, a predefined setpoint $x^{*}$ in minimal time.
In this case $X^{*}=X=\mathbb{R}^{n}$.\\

The cost function $J$, penalizing for time expended on the task,
is \begin{equation}
J(x^{0},u,\tau)=\int_{0}^{\tau}1dt=\tau,\label{eq:time-opt-cost}\end{equation}
where the initial state is $x^{0}$. The plant's linear dynamics are
described by the ODE

\begin{equation}
\begin{cases}
\dot{x}_{t}=Mx_{t}+Nu_{t}\\
x_{0}=x^{0}\end{cases}\qquad;\qquad u_{t}\in[-1,1]^{m}\,,\label{eq:lin-dyn}\end{equation}
where $M$ and $N$ are matrices of dimensions $n\times n$ and $n\times m$
respectively. The results can be generalized to more complicated sets
of controls. We use theorem \ref{thm:meth-have-fm-det} to provide
sufficient conditions for the existence of a forward model in this
case. This is done by showing that linear time optimal control with
delayed state feedback has no NSCTP pairs of states, thereby fulfilling
the necessary conditions of the theorem. The precise statement of
this result is provided in Theorem \ref{thm:lin-time-opt-have-fm}.\\

The proof that the system has no NSCTP pairs of states is based on
geometrical properties of accessible sets, and can be found in section
\ref{sub:Linear-time-optimal}. Using Theorem \ref{thm:lin-time-opt-have-fm},
the need for a forward model in the simple example presented in section
\ref{sub:Simple-example} can be established trivially, since the
matrices $M$ and $N$ are given by $M=-1$ and $N=1$, which leads
to a normal system (a required assumption for theorem \ref{thm:lin-time-opt-have-fm}),
and the set $X=[-1,1]$ satisfies the other assumptions needed.

\subsection{\label{sub:Minimum-Jerk-Optimal}Minimum Jerk Optimal Control}

Many models for the control of human arm movements have been suggested
in an attempt to explain experimental results. The minimum jerk model
was probably the first approach to address these issues based on optimal
control principles \cite{flash84coordination}. In this approach,
a two degree of freedom manipulator endpoint is controlled on a plane
by applying jerk (the third derivative of the position). The task
that the system should perform is taking the plant from some initial
state to a final state in time $T$, minimizing the total accumulated
squared jerk. We show that such a problem, where $T$ is a part of
the task, possesses no NSCTP pairs of states, and therefore by theorem
\ref{thm:meth-have-fm-det}, a CTPS controller based on delayed inputs
must contain a forward model. \\

In this model the state consists of the end-point of the manipulator's
displacement, velocity and acceleration in a plane,

\begin{equation}
\underline{x}^{p}=\left(x,y,\dot{x},\dot{y},\ddot{x},\ddot{y}\right)^{\top}=\left(x,y,u,v,z,w\right)^{\top},\label{eq:min_jerk_sys}\end{equation}
with dynamics\begin{equation}
\dot{\underline{x}}^{p}=\left(\dot{x},\dot{y},\ddot{x},\ddot{y},\delta,\gamma\right)^{\top},\label{eq:min_jerk_dynamics}\end{equation}
where $\delta$ and $\gamma$ are the controls, namely $\underline{u}=\left(\dddot{x},\dddot{y}\right)^{\top}=\left(\delta,\gamma\right)^{\top}.$
We define a task termed \emph{optimal setpoint tracking in constant
time} where the plant must be controlled so that it reaches some state
$x^{p*}$, with zero velocity and acceleration, while optimizing a
cost function $J$, when the initial state of the plant is $x$ and
the time for reaching the goal is $T$ (which is itself part of the
task). Therefore the task is given by $x^{*}=(x^{p*},T)$ and $x^{p*}\in\tilde{X}$,
where \begin{equation}
\tilde{X}=\{x\in X\;|\; u=v=z=w=0\}.\label{eq:res-tildeX}\end{equation}
The cost function is \begin{equation}
J(\underline{x}_{0}^{p},u,T)=\frac{1}{2}\int_{0}^{T}\left(\dddot{x}_{t}^{2}+\dddot{y}_{t}^{2}\right)dt,\label{eq:min_jerk_cost}\end{equation}
with initial conditions $\underline{x}_{0}^{p}=\left(x_{0},y_{0},x_{d0},y_{d0},x_{dd0},y_{dd0}\right)^{\top}$
and boundary conditions $\underline{x}_{T}^{p}=\left(x_{T},y_{T},0,0,0,0\right)^{\top}=\underline{x}^{p*}.$
\\

As was shown in \cite{flash84coordination}, each coordinate, $x$
and $y$, can be computed separately and identically, and the solution
for $x$ has the following form \begin{equation}
x_{t}=a_{0}+a_{1}t+a_{2}t^{2}+a_{3}t^{3}+a_{4}t^{4}+a_{5}t^{5},\label{eq:min-jerk-explicit-sol}\end{equation}
where the constants $a_{i}$ depend on $T$, on the initial conditions
and on $x^{p*}$. Theorem \ref{thm:meth-min-jerk-have-fm} proves
that for this system a forward model is indeed essential. The proof
is based on theorem \ref{thm:meth-have-fm-det} after showing that
the system has no NSCTP pairs of states.\\

Note that when $T$ is constant and is not a part of the control task,
the system has an infinite number of NSCTP pairs, and a similar proof
will not work because it relies on the absence of NSCTP pairs in the
system. However, this does not imply that a forward model is not needed,
but rather that higher order conditions may be required.

\section{\label{sec:methods}Methods and Detailed Proofs }

In this section we rephrase, in a formal mathematical language, the
ideas and results introduced and presented intuitively in section
\ref{sec:Results}. We begin with several technical definitions which
will be required in the sequel.

\subsection{\label{sec:meth-Problem-definition}Basic definitions}

Let $X\subseteq\mathbb{R}^{n}$ be a set of states and $U\subseteq\mathbb{R}^{m}$
the set of possible actions that the controller can choose from. We
use an underline to denote the history of a dynamic variable between
time zero and time $t$, e.g., $\underline{u}_{t}:\;[0,t]\rightarrow U$
and similarly for arbitrary times $[t_{1},t_{2}]$ we use $\underline{u}_{[t_{1},t_{2}]}:\;[t_{1},t_{2}]\rightarrow U$.
Denote by $\mathcal{U}_{t}$ the set of possible piecewise continuous
controls that can be selected up to time $t$, namely $\mathcal{U}_{t}\triangleq\left\{ \underline{u}_{t}:\;\underline{u}_{t}\text{ is piecewise continuous on}\;[0,t]\right\} $.
The plant is given in (\ref{eq:plant_desc}).\\

We introduce a set of tasks to be solved, and a set of controls which
solve these tasks.
\begin{defn}
Let $X^{*}$ be a set of tasks that need to be solved by the controller,
and let $x^{*}$ be a specific task. The set of\emph{ task solving
controls,} $U_{t}^{*}(x^{p},x^{*})$, consists of all piecewise continuous
control laws, in the interval $[0,t]$, that lead to the performance
of task $x^{*}$ when the initial condition is $x^{p}$.
\end{defn}
In the case where the task is completed for $\tau<t$ , the remaining
controls are arbitrary, namely $U_{[\tau,t]}^{*}=\mathcal{U}_{t-\tau}$.
Since we consider situations where the controller executes a series
of tasks, we define the switching task process.
\begin{defn}
\label{def:switching-times}The switching tasks process $z_{t}$ is
defined by $z_{t}\triangleq\sum_{i=0}^{\infty}\delta(t-t_{i}),$ where
$t_{i}$ are the times at which the tasks are switched, and $\delta\left(\cdot\right)$
is the Dirac impulse function.
\end{defn}
The controller is given by (\ref{eq:contr-form}) and its state dynamics
(memory) by (\ref{eq:contr-mem-ode}). While other definitions of
memory may be considered, we limit ourselves in this letter to the
present formulation. We assume that the task definition process $x_{t}^{*}$
is constant between two task switches. It will be convenient in the
sequel to assume that the state space contains all states reachable
for any allowable control law.
\begin{defn}
\label{def:meth-inesc}The set $X\subseteq\mathbb{R}$ is \emph{inescapable}
when for all initial conditions $x_{0}^{p}\in X$, and controls $\underline{u}_{t}\in\mathcal{U}_{t}$,
the state at time $t$ remains in $X$, namely $x_{t}^{p}\in X$.
\end{defn}
In principle, the task solving control laws are not necessarily continuous.
We introduce a subset of continuous control laws.
\begin{defn}
For any $\epsilon>0,$ the set $\tilde{U}_{\epsilon}^{*}(x_{t}^{p},x^{*})\triangleq\{\underline{u}_{\epsilon}\in U_{\epsilon}^{*}(x_{t}^{p},x^{*}):\;\underline{u}_{\epsilon}\,\mathrm{is\, continuous}\}$,
consisting of all continuous task solving controls, is termed the
\emph{continuous task solving control set}.
\end{defn}
Next, we formally introduce the idea of a forward model.
\begin{defn}
A controller possesses a forward model when there exists a transformation
$F$ such that for all times $t$, initial conditions $x_{0}^{p}$,
switching sequences $\underline{z}_{t}$, and tasks $\underline{x}_{t}^{*}$,
the state is given by $x_{t}^{p}=F(x_{t}^{c},\tilde{x}_{t-D}^{p})$
where $\tilde{x}_{t-D}^{p}=x_{(t-D)_{+}}^{p}$.
\end{defn}
\noindent  In section \ref{sub:meth-General-Results} we provide
precise conditions that imply the existence of a forward model.

\subsection{\label{sub:meth-General-Results}General Results}

The present section is constructed as follows. Initially, a system
(plant and controller) with good performance is defined (definition
\ref{def:optimal-system}). We then show that such systems can be
implemented even when the state observation is delayed (Theorem \ref{thm:meth-can-build}).
Finally, whenever the problem is not too trivial (see definition \ref{def:A-pair-of}),
we show that the controller \emph{must} possess a forward model (Theorem
\ref{thm:meth-have-fm-det}).\\

Several assumptions are required before proceeding to the main claims.
We assume that all possible sequences of tasks in $X^{*}$ can be
performed by a controller from any initial condition in $X$, and
we also require that $X$ cannot be escaped by applying legal controls.\\

\begin{assumption}
\label{ass:meth-exist-sol} For each task $x^{*}\in X^{*}$ and initial
state $x^{p}\in X$, a piecewise continuous solution exists, namely,
for any value of $t$, $U_{t}^{*}(x^{p},x^{*})\neq\emptyset$ .
\end{assumption}
\noindent In the sequel we will compare two control laws in a small
interval around $t=0$. In order to do so, based on the values of
the controls at $t=0$, we need to assume the existence of a small
interval over which the task solving controls are continuous. In other
words, for each task $x^{*}\in X^{*}$ and state $x^{p}\in X$, there
exists $\epsilon_{0}(x_{t}^{p},x^{*})>0$ s.t $\tilde{U}_{\epsilon_{0}}^{*}(x_{t}^{p},x^{*})\neq\emptyset$.
The existence of such an interval follows directly from assumption
\ref{ass:meth-exist-sol}.
\begin{assumption}
\label{ass:meth-inescapable}The set $X$ is \emph{inescapable.}~
\end{assumption}
\noindent Given a {}``black box controller'' satisfying certain
conditions, we will demonstrate the existence of a forward model.
\begin{assumption}
\label{ass:meth-piecewise-cont}A task solving {}``black box controller'',
which provides a piecewise continuous and continuous from the right
control signal $u_{t}$, is given.
\end{assumption}
\noindent Next, we define a {}``correctly performing system'',
namely a system which executes all possible sequences of tasks correctly.
\begin{defn}
\label{def:optimal-system}The controller, plant and task space constitute
a \emph{Correct Task Performing System} (CTPS) when for each $\underline{x}_{t}^{*},\, x_{0}^{p},\, t,\underline{\, z}_{t}$,
\[
\underline{u}_{[t_{i},\min(t_{i+1}^{},t)]}\in U_{\min(t_{i+1}^{},t)-t_{i}}^{*}(x_{t_{i}^{}}^{p},x_{t_{i}}^{*})\;\;\mathrm{for\, all}\, i\in\{j:\; t_{j}^{}<t\}.\]
In other words, the controller always selects a signal $u_{t}$ solving
the sequence of tasks.
\end{defn}
At this point we show that \emph{if} there exists a controller without
delay that renders the system CTPS, then a controller with delay can
render the system CTPS as well. The intuitive idea is that in a deterministic
system, the state of the controller can store all past controls, and
thereby simulate the plant in order to predict the current state.
\begin{thm}
\label{thm:meth-can-build}Let $A_{p}$ be a deterministic plant as
in (\ref{eq:plant_desc}) with state variable $x_{t}^{p}$. Then under
assumptions \ref{ass:meth-exist-sol} and \ref{ass:meth-inescapable},
there exists a controller of the form (\ref{eq:contr-form}) such
that the system is CTPS.\end{thm}
\begin{proof}
Define $FM_{D}(x_{0}^{p},\underline{u}_{[0,D]})$ to be the solution
of the dynamics of the plant at time $D$, when the initial state
of the plant is $x_{0}^{p}$ and the control until time $D$ is $\underline{u}_{[0,D]}$.
Now, define the state of the controller at time $t$ to be \begin{equation}
x_{t}^{c}(\alpha)=[t,u_{\alpha}^{*}]^{\top}\in\mathbb{R}^{m+1}\,,\label{eq:meth-build-contr-mem}\end{equation}
where $u_{\alpha}^{*}$ is the correct control at time $\alpha$ for
performing the task. Note the $u_{\alpha}^{*}$ can be defined even
for $\alpha>t$ assuming that $x_{r}^{*}=x_{t}^{*}$ for $t<r<\alpha$
($x^{*}$ does not change). In (\ref{eq:meth-build-contr-mem}) we
separate the first component of the control state (representing time)
from the other components, and use $x_{t,1}^{c}(\alpha)$ for the
former and $x_{t,2}^{c}(\alpha)$ for the remaining $m$-dimensional
sub-vector consisting of $u_{\alpha}^{*}$. We also define a projection
of $x_{t}^{c}$ on an interval $[a,b]$ to be $x_{t}^{c}[a,b]$. The
state $x_{t}^{c}$ defined in (\ref{eq:meth-build-contr-mem}) is
obtained by the dynamics\[
\dot{x}_{t}^{c}(\alpha)=\left[\begin{matrix}1\\
I_{\alpha\geq x_{t,1}^{c}}(u_{\alpha}^{*}-x_{t,2}^{c}(\alpha))z_{t}\end{matrix}\right]\,,\]
where we recall the definition \ref{def:switching-times} of the switching
sequence $\left\{ z_{t}\right\} $ in terms of an impulse train. The
future control is selected from the correct solution set of controls.
More formally, defining $\tilde{D}=\min[x_{t,1}^{c},D]$, we set $\hat{x}_{t}^{p}=FM_{\widetilde{D}}\left(x_{t-\widetilde{D}}^{p},x_{t,2}^{c}[x_{t,1}^{c}-\widetilde{D},x_{t,1}^{c}]\right)$
and $\underline{u}_{[t,\infty)}^{*}\in U_{\infty}^{*}\left(\hat{x}_{t}^{p},x_{t}^{*}\right)$.
In other words $\underline{u}_{[t,\infty)}^{*}$ is chosen from the
correct solution set $U_{\infty}^{*}(\hat{x}_{t}^{p},x_{t}^{*})$
where $\hat{x}_{t}^{p}=x_{t}^{p}$ is the exact prediction of the
current state using the forward model $FM$. For such a definition
of the memory, the control can be chosen by \begin{eqnarray*}
u_{t} & = & B^{c}(x_{t}^{*},z_{t},x_{t}^{c},\tilde{x}_{t-D})\\
 & = & x_{t,2}^{c}\left(x_{t,1}^{c}\right)\,.\end{eqnarray*}
It is obvious that the control between task switches is chosen so
that $u_{[t_{i}^{},t_{i+1}^{}]}$ belongs to the set $U_{t_{i+1}-t_{i}}(x_{t_{i}^{}}^{p},x_{t_{i}^{}}^{*})$
for each $i$, and therefore it is CTPS.
\end{proof}
Next we introduce a property whereby two states may be {}``united''
in terms of the solution to tasks, and therefore cannot be distinguished.
For such states, for each task, there exists a continuous control
such that controls at time $0$ are equal. The absence of such pairs
will enable us to guarantee the existence of a forward model in a
controller.
\begin{defn}
\label{def:A-pair-of}For a problem where assumption \ref{ass:meth-exist-sol}
holds, a pair of distinct states $x^{p}$and $x^{\prime p}$, $x^{p}\neq x^{\prime p}$,
is called a\emph{ Non Separable by Correct Task Performing }pair (NSCTP)
if for all $x^{*}$, and $0<\epsilon<\min\{\epsilon_{0}(x^{p},x^{*}),\epsilon_{0}(x^{\prime p},x^{*})\}$,
there exist controls $u\in\tilde{U}_{\epsilon}^{*}(x^{p},x^{*})$
and $u^{\prime}\in\tilde{U}_{\epsilon}^{*}(x^{\prime p},x^{*})$ such
that $u_{0}=u_{0}^{\prime}$.
\end{defn}
The following theorem constitutes the main theoretical result in the
paper. It provides sufficient conditions for the existence of a forward
model in delayed state feedback control. A required condition is the
absence of NSCTP pairs in the system.
\begin{thm}
\label{thm:meth-have-fm-det}Let $A_{p}$ be a deterministic plant
as in (\ref{eq:plant_desc}), and assume that NSCTP pairs of states
are absent from the system. Let $(x^{c},B_{c})$ be a controller with
delayed state feedback which renders the system CTPS. Then, under
assumptions \ref{ass:meth-exist-sol}, \ref{ass:meth-inescapable}
and \ref{ass:meth-piecewise-cont}, there exists a forward model $F$
such that for each $t,$ any initial condition $x_{0}^{p}$, and history
of tasks $(\underline{x}_{t}^{*},\underline{z}_{t})$, \[
x_{t}^{p}=F(x_{t}^{c},\tilde{x}_{t-D}^{p})\,.\]
\end{thm}
\begin{proof}
Assume that the system is CTPS and assume by negation that such a
forward model $F$ does not exist. Therefore there exist times $t_{1}$
and $t_{2}$, controller states $x_{t_{1}}^{c}(\cdot)=x_{t_{2}}^{c}(\cdot)$,
and plant states $\tilde{x}_{t_{1}-D}^{p}=\tilde{x}_{t_{2}-D}^{p}$
such that for some two trials $(x_{0}^{p},\underline{x}_{t_{1}}^{*},\underline{z}_{t_{1}},t_{1})$
and $(x_{0}^{\prime p},\underline{x}_{t_{2}}^{\prime*},\underline{z}_{t_{2}},t_{2})$
we get $x_{t_{1}}^{p}\neq x_{t_{2}}^{\prime p}$. But the controller
of the form (\ref{eq:contr-form}) chooses the action by the rule
$u_{t}=B^{c}(x_{t}^{*},x_{t}^{c},z_{t},\tilde{x}_{t-D}^{p})$. Therefore,
for each new task $x^{*}$ set at times $t_{1}$ and $t_{2}$ for
the two trials (since the system is inescapable and a solution always
exists), we have $u_{t_{1}}=u_{t_{2}}$ and since $u_{t}$ is continuous
from the right and piecewise continuous, there exist $\epsilon_{0}>0$
such that $u_{t_{i}}$ are continuous on $[t_{i},t_{i}+\epsilon_{i}]$
for $i\in\{1,2\}$. Thus, from the assumption that the system is CTPS,
it follows that for all $x^{*}$ and $0<\epsilon<\min(\epsilon_{0},\epsilon_{1}(x_{t_{1}}^{p},x^{*}),\epsilon_{2}(x_{t_{2}}^{p},x^{*}))$,
$u_{[t_{1},t_{1}+\epsilon]}\in\tilde{U}_{\epsilon}^{*}(x_{t_{1}}^{p},x_{t}^{*})$
and $u_{[t_{2},t_{2}+\epsilon]}\in\tilde{U}_{\epsilon}^{*}(x_{t_{2}}^{p},x_{t}^{*})$
(it suffices to look at the continuous solutions since we know that
the control signal is piecewise continuous for the {}``black box
controller''). But this means that the pair of distinct states $x_{t_{1}}^{p}$
and $x_{t_{2}}^{\prime p}$ is NSCTP which leads to a contradiction
with the assumption that no such states exist.
\end{proof}

\subsection{\label{sub:Linear-time-optimal}Linear time optimal control}

We consider two examples demonstrating the general claims established
in section \ref{sub:meth-General-Results}. We begin in the present
section by considering a linear control problem where the task is
defined as \emph{optimal setpoint tracking, }introduced in section\emph{
\ref{sub:res-Linear-time-optimal}.} The objective here is to minimize
the time required to reach the desired state with linear dynamics
and delayed observations. The formal task is described in definition
\ref{def:meth-optimal-setpoint-tracking}. In order to simplify the
notation, we will omit the superscript $p$ from $x^{p}$ in this
section. Some background results required in this section, and alluded
to below, are taken from \cite{hermes69}.
\begin{defn}
\label{def:meth-optimal-setpoint-tracking} Let $X^{*}=X$ and $x^{*}\in X^{*}$.
The task is an \emph{optimal setpoint tracking} task when \[
U_{t}(x,x^{*})=\argmin_{u,\tau|x_{\tau}=x^{*}}J(x,u,\tau),\]
namely, the controller must take the plant state from the initial
state $x$ to the desired state $x^{*}$ while minimizing the cost
function $J$.
\end{defn}
The time optimal cost function and the dynamics are given in (\ref{eq:time-opt-cost})
and (\ref{eq:lin-dyn}) respectively. Let $u_{[0,t]}$ be a given
control law. Then it is well known that \begin{equation}
x_{t}=X_{t}x^{0}+X_{t}\int_{0}^{t}X_{s}^{-1}Nu_{s}ds,\label{eq:ode_sol}\end{equation}
where the matrix $X_{t}$ is the solution of the system,\[
\dot{X_{t}}=MX_{t}\quad\mathrm{with}\quad X_{0}=I,\]
which can be written explicitly as $X_{t}=e^{tM}$.

The existence of a forward model in this case will be demonstrated
under the following assumptions that are needed to prove the absence
of NSCTP pairs of states, and to fulfill the assumptions of theorem
\ref{thm:meth-have-fm-det}.
\begin{assumption}
\textup{\label{ass:X-is-contr}}The system is essentially normal (as
defined on p. 65 in \cite{hermes69}). The term {}``essentially''
implies that a property holds almost everywhere - except on a set
with measure zero. For simplicity, the term {}``essentially'' will
be omitted from now on in the context of normal systems. The set $X$
is a controllable and inescapable set (see Section \ref{sub:Problem-definition}).
\end{assumption}
The general definition of a normal system is somewhat intricate. However,
for a time independent linear system of the form (\ref{eq:lin-dyn}),
theorem 16.1 in \cite{hermes69} establishes that the system is normal
if and only if for each $j=1,\ldots,m$ the vectors \textcolor{red}{${\color{red}{\color{black}N,MN^{j},\ldots,M^{n-1}N^{j}}}$
}are linearly independent, where $N^{j}$ are the column vectors of
the matrix $N$. The exact conditions on the matrices $M$ and $N$
needed for the set $X$ to be controllable and inescapable require
further analysis. However, a condition such as stability of $M$ insures
the existence of a set $X$ with $0\in X$, that will be both controllable
and inescapable. \\

As stated above, the main results in the present section rely heavily
on basic concepts and theorems from \cite{hermes69}. For ease of
reference, we recall some basic notions.
\begin{defn}
Let $K(t,x^{0})$ be the \emph{accessible set }at time $t$, starting
from $x^{0},$namely \[
K(t,x^{0})\triangleq\{x:\mbox{\ensuremath{\,} there exists }u\mbox{ which steers from }x^{0}\mbox{ to }x\mbox{ at time \mbox{t}\}.}\]

\end{defn}
The following two key observations about normal systems are taken
from \cite{hermes69}.
\begin{itemize}
\item For a normal system, $K(t,x^{0})$ is strictly convex, bounded and
closed.
\item For normal systems, an optimal control law always exists, is unique
and is essentially determined by $u_{t}^{*}=\mathrm{sgn}(\eta^{T}X_{\tau^{*}}X_{t}^{-1}N)$
for $x^{0},x^{*}\in X$ , where $\eta$ is an outward normal to $K(t,x^{0})$
at $x^{*}$, and the trajectory $x_{[0,t]}^{*}$ is unique.
\end{itemize}
We begin by proving a basic lemma that establishes some properties
that are required in order to show that the system does not possess
NSCTP pairs of states. The lemma establishes geometric properties
of two intersecting accessible sets. A sketch of the ideas underlying
the lemma is presented in Figure \ref{fig:Accessible-sets-meet}.
\begin{lem}
\label{lem:find-x-star}Let $x^{1},x^{2}\in X$ and $x^{1}\neq x^{2}$,
and define $\tau^{m}\triangleq\sup\left\{ \tau:\; K(\tau,x^{1})\cap K(\tau,x^{2})=\emptyset\right\} $.
Then under assumption \ref{ass:X-is-contr}:\end{lem}
\begin{enumerate}
\item $\tau^{m}<\infty$.
\item $K(\tau^{m},x^{1})\cap K(\tau^{m},x^{2})=\{x^{*}\}$.
\item There exists an outward normal $g$ to a supporting hyperplane to
$K(\tau^{m},x^{1})$ at $x^{*}$ and $-g$ is an outward normal to
a supporting hyperplane to $K(\tau^{m},x^{2})$ at $x^{*}$.

\begin{center}
Figure 3 about here
\end{center}

\end{enumerate}
\begin{proof}
For a normal system there exists an optimal control, namely for all
$x^{0},x^{*}\in X$ there exists a $\tau$ (might be infinity) such
that $x^{*}\in\partial K(\tau,x^{0})$ (by Theorems 14.1, 14.2, 15.1
and Corollary 15.1 in \cite{hermes69}). Define $L\triangleq K(\tau^{m},x^{1})\cap K(\tau^{m},x^{2})$.
\\
\emph{}\\
\emph{Proof of 1:} Assume by negation that $\tau^{m}=\infty$.
We know also that $X=\lim_{\tau\rightarrow\infty}K(\tau,x^{1})=\lim_{\tau\rightarrow\infty}K(\tau,x^{2})$
from assumption \ref{ass:X-is-contr}. From the definition of $\tau^{m}$,
under assumption that $\tau^{m}=\infty$ there must exist $x^{3}\in X^{\circ}$
such that for all $\tau<\infty$, $x^{3}\notin K(\tau,x^{1})$ (without
loss of generality). If $\tau^{m}=\infty$, then there is no optimal
control from $x^{1}$ to $x^{3}$ which contradicts the existence
of time optimal solution. Therefore $\tau^{m}<\infty$.\\
\\
\emph{Proof of 2: }First let us show that $L\neq\emptyset$. Assume
by negation that $K(\tau^{m},x^{1})\cap K(\tau^{m},x^{2})=\emptyset$.
Since $K$ is closed, strictly convex and compact (from Lemma 12.1,
Corollary 15.1 in \cite{hermes69} and $\tau^{m}<\infty$), the sets
$K(\tau^{m},x^{1}),K(\tau^{m},x^{2})$ are strictly separable by a
hyperplane $f(x)=a\cdot x+b$ i.e., there exists $\epsilon>0$ such
that for all $y\in K(\tau^{m},x^{1})$ $f(y)<-\epsilon$, and for
all $z\in K(\tau^{m},x^{2})$, $f(z)>\epsilon$ by Proposition 2.4.3
\cite{bertsekas_opt}. Define $\tau^{1}=\inf\{\tau:\; K(\tau,x^{1})\cap\{x:f(x)=-\epsilon\}\neq\emptyset\}$
and $\tau^{2}=\inf\{\tau:\; K(\tau,x^{2})\cap\{x:f(x)=\epsilon\}\neq\emptyset\}$.
Notice that $K(\tau^{1},x^{1})\cap\{x:f(x)=-\epsilon\}$ contains
at most a single point since $K$ is closed, strictly convex and an
optimal control always exists. The same argument applies to $K(\tau^{2},x^{2})\cap\{x:f(x)=+\epsilon\}$.
Therefore $K(\tau^{1},x^{1})\cap K(\tau^{2},x^{2})=\emptyset.$ It
follows that also for $\tau^{0}=\min(\tau^{1},\tau^{2})$, we have
that $K(\tau^{0},x^{1})\cap K(\tau^{0},x^{2})=\emptyset$. But $\tau^{0}>\tau^{m}$,
and this contradicts the definition of $\tau^{m}$, therefore $L\neq\emptyset$.\\
Next we show that $L^{\circ}=\emptyset$ . Assume by negation that
it is not and let $x\in L^{\circ}$. Therefore $x\in K^{\circ}(\tau^{m},x^{1})$
and $x\in K^{\circ}(\tau^{m},x^{2})$, but from the definition of
$\tau^{m}$ for all $\epsilon>0$, $x\notin K(\tau^{*}-\epsilon,x^{1})$
or $x\notin K(\tau^{m}-\epsilon,x^{2})$ (without the loss of generality
assume that $x\notin K(\tau^{m}-\epsilon,x^{1})$) and $K$ is monotonic
in $\tau$. Therefore for all $M>0$, $x\in K^{\circ}(\tau^{m}+M,x^{1})$
which leads to a contradiction that there is no optimal control from
$x^{1}$ to $x$ as should be by Theorem 15.1 and Corollary 15.1 in
\cite{hermes69}. Therefore $L^{\circ}=\emptyset$.\\
Let us show that $L$ cannot include more than a single point.
Since $L^{\circ}=\emptyset$ and $L$ is strictly convex, then if
$x,x^{\prime}\in L$ and $x\neq x^{\prime}$ then a convex combination
should be in $L$. But since $L$ is strictly convex , the convex
combination cannot be on $\partial L$ or in the interior of $L$
since it is empty. Therefore $L$ can contain only a single point.\\
Summarizing the above, $L$ is not empty and can contain only a
single point, therefore $L=\{x^{*}\}$.\\
\\
\emph{Proof of 3:} Define $K_{1}\triangleq K(\tau^{m},x^{1})$
and $K_{2}\triangleq K(\tau^{m},x^{2})$. Since $K_{1}^{\circ}\cap K_{2}^{\circ}=\emptyset$
and $K_{1},K_{2}$ are convex, we can use the separating theorem for
$K_{1}^{\circ},K_{2}^{\circ}$ (Proposition 2.4.2 \cite{bertsekas_opt}).
Thus there exists a $g\in\mathbb{R}^{n}$, $g\neq0$, such that for
all $x\in K_{1}^{\circ}$, $x^{\prime}\in K_{2}^{\circ}$, $g\cdot x\leq g\cdot x^{\prime}$.
Now let $x_{n}\in K_{2}^{\circ}$ such that $x_{n}\rightarrow x^{*}$
thus $g\cdot(x-x_{n})\leq0$ and therefore $g\cdot(x-x^{*})\leq0$.
Similarly, we find that $g\cdot(x^{\prime}-x^{\ *})\geq0$. Since
the functional $(g\cdot x)$ is continuous, the same is correct for
$y\in K_{1},y^{\prime}\in K_{2}$, i.e $g\cdot(y-x^{*})\leq0$ and
$g\cdot(y^{\prime}-x^{\ *})\geq0$. Thus $g,-g$ are outward normals
to supporting hyperplanes $K_{1},K_{2}$ respectively.
\end{proof}
Using lemma \ref{lem:find-x-star} we will establish that the system
does not possess NSCTP pairs of states, and thus the need for a forward
model will follow from theorem \ref{thm:meth-have-fm-det}.
\begin{thm}
\label{thm:lin-time-opt-have-fm}Consider a linear normal system described
by (\ref{eq:lin-dyn}), and assume that a controller with delayed
input renders the system CTPS for an optimal setpoint tracking task,
where the cost function is given by (\ref{eq:time-opt-cost}). If
$x_{t}^{c}$ is the memory state of the controller, and assumptions
\ref{ass:meth-piecewise-cont} and \ref{ass:X-is-contr} hold, then
there exists a forward model $F$ such that for each $t,$ any initial
condition $x_{0}^{p}$, and history of tasks $(\underline{x}_{t}^{*},\underline{z}_{t})$,
\[
x_{t}^{p}=F(x_{t}^{c},\tilde{x}_{t-D}).\]
\end{thm}
\begin{proof}
First notice that assumption \ref{ass:meth-piecewise-cont} holds
since we required the system to be controllable, and from Theorem
13.1 in \cite{hermes69}, the minimizer exists. Thus the task can
always be performed, and from the normality of the system it follows
that the time optimal control reaching $x^{*}$ is bang-bang, which
implies that $\tilde{U}_{\epsilon}$ is not empty. First we will show
that the system has no NSCTP pairs of states, and then use theorem
\ref{thm:meth-have-fm-det} to establish the existence of a forward
model.

For a normal system, the time optimal control reaching $x^{*}$ is
given by\[
u_{t}=\mathrm{sgn}(\eta^{T}X_{\tau^{*}}X_{t}^{-1}N),\]
where $\eta$ is an outward normal to a supporting hyperplane to $K(\tau^{*},x^{0})$
at $x^{*}$ (except on a set of measure 0). It is essentially unique
(may differ over a set of times with measure 0) by Theorems 14.1,
14.2, 15.1 and Corollary 15.1 in \cite{hermes69}. Now, let $x^{1}$
and $x^{2}$ be two distinct points in $X$, then by lemma \ref{lem:find-x-star}
there exists $x^{*}$ which is reachable from $x^{1}$ and $x^{2}$
in time $\tau^{*}=\tau^{m}$ (since $x^{*}\in\partial K(x^{1},\tau^{m})$
and $x^{*}\in\partial K(x^{2},\tau^{m})$) by time optimal control,
and there exist outward normals $\eta_{1}=g$ and $\eta_{2}=-g$.
Since $X_{t}$ does not depend on the initial conditions,\[
u_{t}^{1}=-u_{t}^{2}\neq0,\]
$u_{t}$ is piecewise continuous and $u_{0}^{1}\neq u_{0}^{2}$. Therefore
for an arbitrary $x^{1}$ and $x^{2}$ we have found $x^{*}$ such
that the solution is unique and $u_{0}^{1}\neq u_{0}^{2}$ . Thus
the system does not possess NSCTP pairs of states. We have shown that
all the assumptions required for theorem \ref{thm:meth-have-fm-det}
hold, and therefore there exists a forward model $F$ such that $x_{t}^{p}=F(x_{t}^{c},\tilde{x}_{t-D}).$
\end{proof}

\subsection{\label{sub:meth-Minimum-Jerk-Optimal}Minimum Jerk Optimal Control}

In this example, the plant's state, dynamics, control and cost functions
are given in (\ref{eq:min_jerk_sys}-\ref{eq:min_jerk_cost}). The
initial and terminal conditions are given in section \ref{sub:Minimum-Jerk-Optimal}.
The solution trajectory is given in (\ref{eq:min-jerk-explicit-sol}),
where the constants $a_{i}$ are found using the initial and boundary
conditions. Taking three derivatives of (\ref{eq:min-jerk-explicit-sol})
and setting $t=0$, we obtain\begin{eqnarray*}
\delta_{0} & = & \frac{60}{T^{3}}x_{T}-\frac{60}{T^{3}}x_{0}-\frac{36}{T^{2}}x_{d0}-\frac{9}{T}x_{dd0}.\end{eqnarray*}
First, notice that for a constant value of $T$, there exist NSCTP
pairs. Each $\underline{x}=\left(x_{0},x_{d0},x_{dd0}\right)^{\top}$
and $\underline{x}^{\prime}=\left(x_{0}^{\prime},x_{d0}^{\prime},x_{dd0}^{\prime}\right)^{\top}$
such that\[
-\frac{60}{T^{3}}x_{0}-\frac{36}{T^{2}}x_{d0}-\frac{9}{T}x_{dd0}=-\frac{60}{T^{3}}x_{0}^{\prime}-\frac{36}{T^{2}}x_{d0}^{\prime}-\frac{9}{T}x_{dd0}^{\prime}\]
are NSCTP pairs (there are infinitely many of these) since for each
$x_{T}$ the optimal control at time $0$ is\[
\delta_{0}(x_{T})=\delta_{0}^{\prime}(x_{T}^{}).\]
This result does not imply that in the present case a forward model
is not needed, but it does imply that a higher order condition may
be required in order to prove it.\\

Assume then that the terminal time $T$ can vary. For this case we
will prove in theorem \ref{thm:meth-min-jerk-have-fm} that a forward
model is essential. First we will show in Lemma \ref{lem:min_jerk_no_non_sep}
that the system does not have NSCTP pairs of states, and then use
theorem \ref{thm:meth-have-fm-det} to establish the claim.\\

\begin{defn}
\label{def:meth-optimal-setpoint-tracking-const-time} Let $X^{*}=\widetilde{X}\times$$\mathbb{R}$
, where $\tilde{X}\subseteq X$, and $x^{*}=\left(x^{p*},T\right)\in X^{*}$.
The task is an \emph{optimal setpoint tracking in constant time} task
when \[
U_{t}^{*}(x,x^{*})=\argmin_{u|x_{T}=x^{p*}}J(x,u,T),\]
namely, the controller must take the plant state from the initial
state $x$ to the desired state in time $T$, while minimizing the
cost function $J$.
\end{defn}
In the present case the subset $\tilde{X}$ is given by (\ref{eq:res-tildeX}).
\begin{lem}
\label{lem:min_jerk_no_non_sep}The system (\ref{eq:min_jerk_sys})
with dynamics (\ref{eq:min_jerk_dynamics}) solving an \emph{optimal
setpoint tracking in constant time} task, and minimizing the cost
function (\ref{eq:min_jerk_cost}) has no NSCTP pairs of states.\end{lem}
\begin{proof}
First, the control $(\delta,\gamma)$ is continuous, therefore $\tilde{U}_{T}^{*}(x_{t}^{p},x^{*})\subseteq U_{T}^{*}(x_{t}^{p},x^{*})$.
The solutions are unique, therefore we just have to find a task $x^{*}$
where the controls at time 0 are different for 2 initials states.

Let $\underline{x}_{0}\neq\underline{x}_{0}^{\prime}$ be two initial
states. Assume, without loss of generality, that the $x$ coordinate's
initial conditions are different in the two initial states, i.e.,
$\underline{x}=\left(x_{0},x_{d0},x_{dd0}\right)$ and $\underline{x}^{\prime}=\left(x_{0}^{\prime},x_{d0}^{\prime},x_{dd0}^{\prime}\right)^{\top}$
such that $\underline{x}\neq\underline{x}^{\prime}$. To show that
the states are not a NSCTP pair we have to find $T$ and $x_{T}$
such that $\delta_{0}^{*}(x_{T})\neq\delta_{0}^{\prime*}(x_{T})$,
where $\delta_{t}^{*}(x_{T})$ and $\delta_{t}^{\prime*}(x_{T})$
are the optimal controls to $x_{T}$ from the initial states $\underline{x}$
and $\underline{x}^{\prime}$ respectively. Recall that the optimal
control at time $0$ is given by\[
\delta_{0}^{*}=\frac{60}{T^{3}}x_{T}-\frac{60}{T^{3}}x_{0}-\frac{36}{T^{2}}x_{d0}-\frac{9}{T}x_{dd0.}\]
The necessary and sufficient condition for equality of the controls
$\delta_{0}^{*}(x_{T})=\delta_{0}^{\prime*}(x_{T})$ is\[
9(x_{dd0}-x_{dd0}^{\prime})T^{2}+36(x_{d0}-x_{d0}^{\prime})T+60(x_{0}-x_{0}^{\prime})=0.\]
Since this is a second order polynomial in $T$, there can be at most
2 roots $T_{1}$ and $T_{2}$. Let $\tilde{T}\neq T_{1},T_{2}$ and
let $x_{\tilde{T}}$ be an arbitrary position, thus for $\tilde{T},x_{\tilde{T}}$,
$\delta_{0}^{*}\neq\delta_{0}^{\prime*}$ which means that the pair
of states $\underline{x}_{0}\neq\underline{x}_{0}^{\prime}$ are not
a NSCTP pair.
\end{proof}
At this point we are ready to prove the existence of a forward model.
\begin{thm}
\label{thm:meth-min-jerk-have-fm}A {}``black box controller'' with
delayed state feedback fulfilling assumption \ref{ass:meth-piecewise-cont}
which renders system (\ref{eq:min_jerk_sys}) with dynamics (\ref{eq:min_jerk_dynamics})
CTPS for an optimal setpoint tracking in constant time task with cost
function (\ref{eq:min_jerk_cost}), must possess a forward model.
In other words, there exists a forward model $F$ such that for each
$t,$ any initial condition $x_{0}^{p}$, and history of tasks $(\underline{x}_{t}^{*},\underline{z}_{t})$,\[
x_{t}^{p}=F(x_{t}^{c},\tilde{x}_{t-D}^{p}).\]
\end{thm}
\begin{proof}
First, assumptions \ref{ass:meth-exist-sol} and \ref{ass:meth-inescapable}
hold trivially since the optimal trajectory is unique and continuous,
and there exists a polynomial solution for each $x_{0}^{p}\in X$
and $x^{*}\in X^{*}$. From Lemma \ref{lem:min_jerk_no_non_sep} we
have that the system does not have NSCTP pairs, and therefore by Theorem
\ref{thm:meth-have-fm-det} there exists a forward model $F$ such
that for each $x_{0}^{p},\underline{x}_{t}^{*},t$\[
x_{t}^{p}=F(x_{t}^{c},\tilde{x}_{t-D}^{p}).\]

\end{proof}

\section{Discussion\label{sec:Discussion}}

We have studied the general problem of control based on delayed state
observations. For this purpose we have formalized the notion of a
system solving a set of control tasks, which is general enough to
cover many of the standard control settings such as regulation and
tracking. Under rather mild conditions on the system, we have shown
that such a controller \emph{must} contain within itself a forward
model. This implies that the current plant state can be exactly determined
based on the delayed state observation and the internal controller
state. We applied our general framework to two widely studied problems,
linear time optimal control and minimum jerk control, and provided
explicit conditions for the necessity of a forward model. These results,
and the general framework itself, provide powerful mathematical support
for the existence of forward models in biological motor control, and,
in fact, in any control system with delayed feedback. \\

A possible limitation of our approach is its restriction to deterministic
systems, as the notion of a forward model used here is clearly inapp\label{fig:Accessible-sets-meet-1}licable
in a stochastic setting. Since in a stochastic setting one cannot
determine the state precisely, a reasonable requirement in this case
is that the posterior state distribution, based on the observed delayed
state and on previous controls, be determined from the present controller
state. As was shown in \cite{AltNai92}, for additive cost functions
the problem of control with delayed observations can be expressed
as a Markov decision process without delay of a more complicated system.
While we have obtained some results in this more challenging and realistic
setting, the full elaboration of this issue is left for future work.
A further open issue relates to approximate, rather than exact, task
performance. We expect that in this case some notion of approximate
forward model will play a role (e.g., \cite{AndSonIgl08}).\\

An interesting question relates to the necessity of the conditions
we have provided, as we have only shown them to be sufficient. In
fact, it is quite possible that milder conditions than the absence
of NSCTP pairs suffice. Finally, it would clearly be of significant
value to demonstrate the absence of NSCTP pairs, and thus the necessity
of forward models, in more biologically relevant settings. However,
proving this for nonlinear dynamical systems, with a level of complexity
approaching that of biological systems, may require non-trivial analysis.
We hope that simpler and mathematically more tractable conditions
can be developed, whose existence will be easier to demonstrate.

\bibliographystyle{plain}
\bibliography{ForwardModel}

\newpage

\subsection*{Figure Captions}

\bigskip

\noindent \textbf{Figure 1} A delayed feedback control system, where
the delayed plant state $x_{t}^{p}$ is observed by a controller.
The sequence $x_{t}^{*}$ represents a set of tasks, and the sequence
$z_{t}$ denotes the times at which tasks are switched.\\

\noindent \textbf{Figure 2} A simple one-dimensional example where
$X^{*}=X=[-1,1],\, U=[-1,1]$, and the objective is to drive the system
to the point $x^{*}.$ The exact control solution in this case is
$\mathrm{sgn}\left(x^{*}-x_{t}\right)$\\

\noindent \textbf{Figure 3} Two accessible sets meet: The sets $K\left(\tau^{m},x^{1}\right)$
and $K\left(\tau^{m},x^{2}\right)$ intersect at time $\tau^{m}$
with the point $x^{*}$ at the intersection with the outward normals
to the support hyperplane.

\newpage

\begin{figure}[h!]
\begin{centering}
\includegraphics[width=0.5\textwidth]{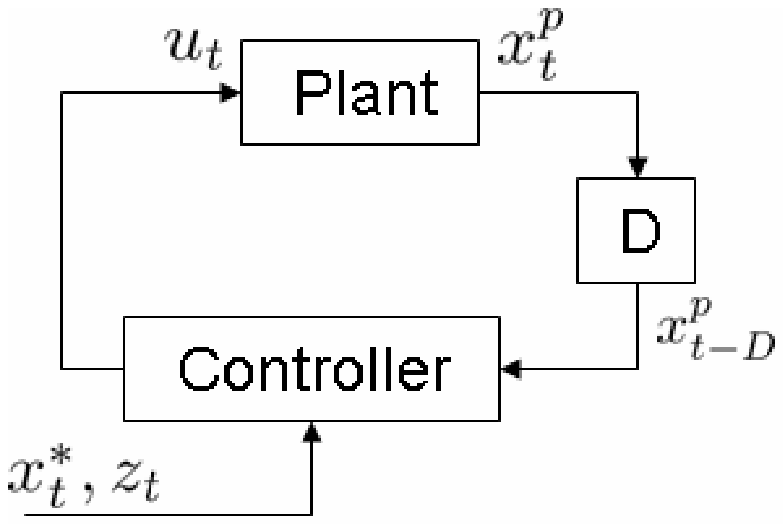}
\end{centering}
\caption{\label{fig:The-system}}
\end{figure}

\begin{figure}[h!]
\begin{centering}
\includegraphics[width=0.5\textwidth]{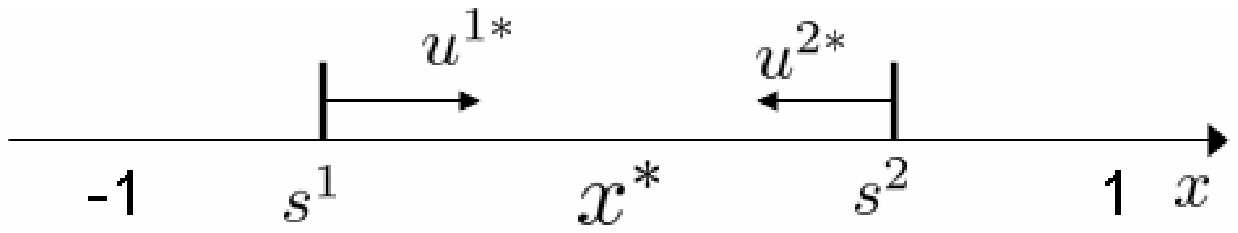}
\end{centering}
\caption{\label{fig:Simple-Example}}
\end{figure}

\begin{figure}[h!]
\begin{centering}
|
\includegraphics[width=0.5\textwidth]{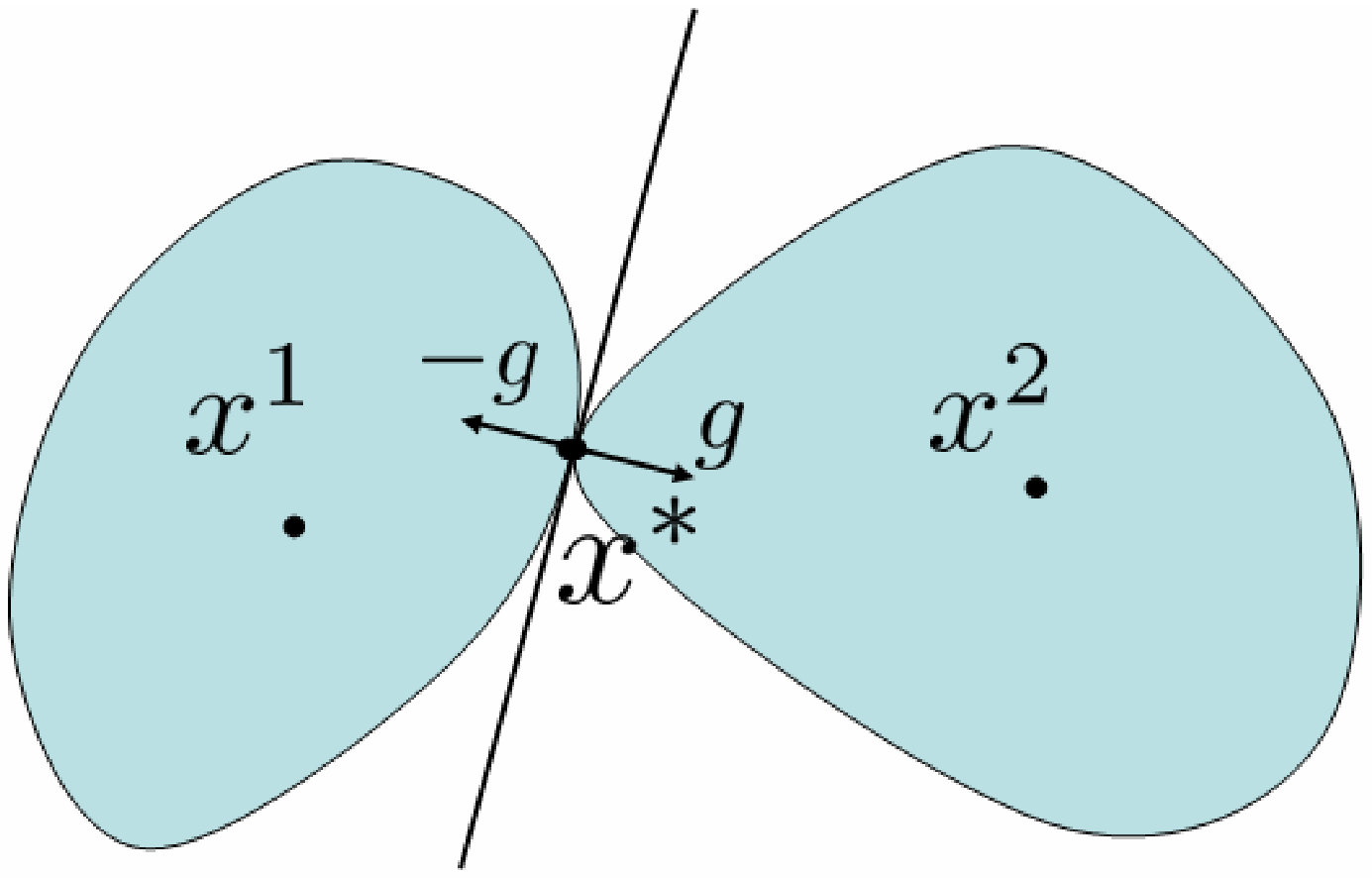}
\end{centering}
\caption{\label{fig:Accessible-sets-meet}}
\end{figure}

\end{document}